\documentclass[12pt]{amsart}

\usepackage{amssymb}
\usepackage{amscd}
\usepackage[all,tips]{xy}
\usepackage[bookmarks=true]{hyperref}
\hypersetup{colorlinks,citecolor=blue,linkcolor=blue}

\newtheorem{thm1}{Theorem}
\newtheorem{theorem}{Theorem}[section]

\newtheorem{prop}[theorem]{Proposition}
\newtheorem{cor}[theorem]{Corollary}
\newtheorem{conjecture}[theorem]{Conjecture}
\newtheorem{problem}[theorem]{Problem}
\newtheorem*{claim*}{Claim}
\theoremstyle{definition}

\newtheorem{rmk}[theorem]{Remark}

\newcommand\Hy{\mathbb{H}}
\newcommand\Sph{\mathbb{S}}

\newcommand\Z{\mathbb{Z}}

\newcommand\R{\mathbb{R}}
\newcommand\C{\mathbb{C}}

\newcommand\Orb{\mathcal{O}}

\newcommand\cT{\mathcal{T}}

\newcommand\rank{\mathrm{rank}}
\newcommand\vol{\mathrm{vol}}
\newcommand\dist{\mathrm{dist}}

\DeclareMathOperator{\PSL}{PSL}

\begin{document}
\title{Arithmetic Kleinian groups generated by elements of finite order}


\author{Mikhail Belolipetsky}\thanks{Belolipetsky is partially supported by CNPq and FAPERJ research grants.}
\address{
IMPA\\
Estrada Dona Castorina, 110\\
22460-320 Rio de Janeiro, Brazil}
\email{mbel@impa.br}

\begin{abstract}

We show that up to commensurability there are only finitely many cocompact arithmetic Kleinian groups generated by rotations. This implies, in particular, that there exist only finitely many conjugacy classes of cocompact two generated arithmetic Kleinian groups. The proof of the main result is based on topological expansion properties of maximal arithmetic $3$-orbifolds. 
We generalize an inequality of Gromov and Guth and combine it with the new bounds for the isoperimetric constant, hyperbolic volume and tube volume of the quotient orbifolds.

\end{abstract}

\maketitle

\section{Introduction}

A \emph{Kleinian group} is a discrete subgroup of the group $\PSL(2, \C)$, which is the full group of orientation preserving isometries of the hyperbolic $3$-space $\Hy^3$. A discrete subgroup of finite covolume is called a \emph{lattice}. An important class of Kleinian groups which are lattices is formed by arithmetic Kleinian groups. 

The most basic classes of Kleinian groups are the groups generated by two elements and closely related to them groups generated by three involutions (see \cite{Gil97, Bel02}). These groups in general and arithmetic Kleinian groups of these types in particular were extensively studied by many authors (see e.g. \cite{Kli89, HLM91, GM94, GMMR97, Gil97, GMM98, MM99, CMMB02, Bel02, KK05}). Still it was not known so far, for instance, whether or not there are finitely many conjugacy classes of \mbox{$2$-generator} arithmetic Kleinian groups. Our first observation in this paper is that for the finiteness problems it is more convenient to consider the class $\mathcal{C}$ of \emph{all} lattices generated by elements of finite order without any restriction on the number of generators. We shall call a lattice in $\mathcal{C}$ \emph{maximal} if it is not properly contained in any other lattice from $\mathcal{C}$. Our main result is the following theorem.

\begin{thm1}\label{thm1}
There exist only finitely many conjugacy classes of maximal cocompact arithmetic Kleinian groups generated by elements of finite order. 
\end{thm1}

Here are some immediate corollaries of Theorem~\ref{thm1}.

\begin{cor}\label{cor1}
There are only finitely many commensurability classes of cocompact arithmetic Kleinian groups which contain groups generated by elements of finite order. 
\end{cor}

This result follows from the fact that every such commensurability class will contain at least one maximal group from $\mathcal{C}$.
Note that if a Kleinian group $\Gamma_0$ is generated by elements of finite order then the orbifold $\Hy^3/\Gamma_0$ has the first Betti number equal to zero, and thus can be called an orbifold rational homology sphere. Therefore, our corollary provides a partial answer to Question~6.1 from \cite{LMR06}, which asks about finiteness of the number of commensurability classes of these spheres. The conjectural answer to the Long--Maclachlan--Reid question is negative so Corollary~\ref{cor1} and also more general Theorem~\ref{thm2.1} imply that we should expect that most of the orbifold arithmetic rational homology spheres are non-commensurable with the compact hyperbolic $3$-orbifolds whose underlying space is an elliptic manifold.

\begin{cor}\label{cor2}
There are only finitely many conjugacy classes of cocompact torsion-free $2$-genera\-tor arithmetic Kleinian groups. 
\end{cor}

To prove this corollary we recall that any $2$-generator Kleinian group is contained with index at most $2$ in a Kleinian group generated by three involutions (half-turns) \cite{Bel02, Gil97}. Hence by Corollary~\ref{cor1} there are only finitely many commensurability classes of such groups. It turns out that there are only finitely many torsion-free $2$-generator groups in a commensurability class (by \cite{BS11}), hence the corollary follows.
We recall that a result of this kind was previously proved in \cite{BS11} following ideas of Agol, but their proof is conditional on a deep open conjecture about short geodesics of arithmetic $3$-manifolds. 

In fact, it is possible to use our results to prove an even stronger finiteness theorem which covers all $2$-generator cocompact arithmetic Kleinian groups (see Theorem~\ref{thm_2gen}). This requires an extension of the methods from \cite{BS11} to the orbifold fibrations. We recall that finiteness of arithmetic Kleinian groups generated by two elliptic elements was previously established in \cite{MM99} by a different method but the general finiteness problem remained open. It is worth mentioning that there is no analogous property for the groups generated by more than two elements. For example, consider an orbifold $\Orb$ obtained by a $(2,0)$-Dehn filling of the figure-$8$ knot complement. This orbifold fibers over a circle with fiber a torus with one order $2$ singularity, which gives a \mbox{$3$-generator} presentation of its fundamental group. It is well known and not hard to check that $\Orb$ is arithmetic. Therefore, the $m$-fold cyclic coverings of $\Orb$ induced from the fibration give rise to an infinite sequence of commensurable compact arithmetic $3$-orbifolds with $3$-generator fundamental groups. Examples of this kind can be also produced for a bigger number of generators. 

\bigskip

The proof of Theorem~\ref{thm1} is based on a generalized Gromov--Guth inequality:
\begin{equation}\label{sec1:eq1}
\lambda_1^{\frac{n}{n-1}} V_{hyp}^{\frac{n}{n-1}} \le c_{n} T^{-n} V_T,
\end{equation}
where $\lambda_1$ is the first non-zero eigenvalue of the Laplacian of a manifold cover of $\Orb = \Hy^3/\Gamma$, $V_{hyp}$ is the hyperbolic volume of $\Orb$ and $V_T$ is the volume of a tube of radius $T$ of a thick embedding of $\Orb$ in $\R^n$. One can notice a similarity of this approach to \cite{LMR06}, and even more to \cite{Ag06}, where an orbifold version of the Li--Yau inequality for conformal volume was applied to show the finiteness of arithmetic maximal reflection groups of isometries of $\Hy^3$. One of the differences between the inequality \eqref{sec1:eq1} and the Li--Yau type inequalities is that the hyperbolic volume in \eqref{sec1:eq1} appears with the exponent $\frac{n}{n-1} > 1$, which is strictly bigger than the exponent of the tube volume $V_T$. In the proof of \eqref{sec1:eq1} this comes out as a consequence of Falconer's slicing inequality (see Section~\ref{sec3} and \cite{GG12}). This property will be crucial for our application of the inequality \eqref{sec1:eq1} in Section~\ref{sec7}.

We are going to apply the inequality \eqref{sec1:eq1} to a quotient orbifold $\Orb$ of a maximal arithmetic Kleinian group $\Gamma \ge \Gamma_0$, with $\Gamma_0\in \mathcal{C}$. 
The left hand side of \eqref{sec1:eq1} depends on the hyperbolic structure of the orbifold $\Orb$ while the right hand side is essentially determined by the topology of its underlying and singular sets. This brings forward a problem of choosing some invariant of $\Orb$ that would connect the two sides. Our first attempts were focused on arithmetic invariants, torsion in homology and Matveev's complexity. This can be compared with \cite{ABSW}, where the arithmetic invariants were used in a Li--Yau type inequality in order to prove the finiteness of arithmetic maximal reflection groups. It is possible to gain a control of these parameters but unfortunately the bounds that we were able to obtain were exponentially far from the required range. The final choice which appears in the argument below is the size of certain bounded degree triangulations of $\Orb$. Such triangulations are constructed in Section~\ref{sec5} using arithmeticity and then applied in Section~\ref{sec6} to bound the complexity of the underlying space of $\Orb$. Along with these considerations, we recall the results of Gromov and Guth about retraction thickness in order to control the singular structure of $\Orb$. Finally, the well known properties of the spectrum of the Laplacian and volumes of arithmetic $3$-orbifolds 
combined with the generalized Buser and Cheeger inequalities provide us a required control over $h$ and $V_{hyp}$. We shall describe the contents of the paper in more detail after the first reductions in the next section. 

The proofs of the results are effective but we shall not investigate their quantitative side in this paper.

\medskip

\noindent\textbf{Convention.} We shall use letters $c$ with indexes to denote the positive constants. When a constant depends on some parameters, they will be specified in the index. All logarithms in the paper are taken to base $2$.

\section{First reductions and an overview of the proof} \label{sec2}

Let $\Gamma_0 < \PSL(2, \C)$ be a maximal Kleinian group generated by elements of finite order, i.e. $\Gamma_0$ is a maximal lattice in $\mathcal{C}$. By Armstrong's theorem \cite{Arm68}, the  underlying space $|\Hy^3/\Gamma_0|$ of the quotient hyperbolic $3$-orbifold is simply connected. Hence by Perelman's proof of the Poincar\'e conjecture it is diffeomorphic to the \mbox{$3$-sphere} $\Sph^3$. Let $\Gamma > \Gamma_0$ be a maximal lattice containing $\Gamma_0$. Since $\Gamma_0$ is a maximal lattice generated by elements of finite order, it is a normal subgroup of $\Gamma$ and the quotient finite group $G = \Gamma/\Gamma_0$ acts without fixed points on $\Hy^3/\Gamma_0$. It follows that the underlying space of the orbifold $\Hy^3/\Gamma$ is an elliptic $3$-manifold $\Sph^3/G$. 
Moreover, for a future reference we remark that the group $\Gamma_0$ is uniquely determined by $\Gamma$ (indeed, if $\Gamma$ would contain two different maximal Kleinian groups $\Gamma_1, \Gamma_2 \in \mathcal{C}$, then their product is also a Kleinian group $\Gamma_1\Gamma_2 \in \mathcal{C}$, which leads to a contradiction with the maximality condition for $\Gamma_i$).
We are going to prove the following more general result which implies Theorem~\ref{thm1} as a corollary.

\begin{theorem}\label{thm2.1}
There exist only finitely many maximal cocompact arithmetic Kleinian groups $\Gamma$ such that the underlying space of the orbifold $\Hy^3/\Gamma$ is an elliptic $3$-manifold.  
\end{theorem}

The proof of Theorem~\ref{thm2.1} is based on a study of geometry of arithmetic hyperbolic $3$-orbifolds and their underlying spaces. More precisely, the arithmeticity assumption is necessary only in some parts of the argument (indicated below) but as it is confirmed by the well known examples, it cannot be removed from the statement. 

Let us recall the definition of an arithmetic Kleinian group. Let $k$ be a number field with exactly one complex place, $R$ its ring of integers, and $A$ a quaternion algebra over $k$ ramified at all real places of $k$. Let $\mathfrak{D}$ be a maximal $R$-order of $A$, denote by $\mathfrak{D}^1$ its group of elements of norm $1$. 
Consider a $k$-embedding $\rho: A \hookrightarrow \mathrm{M}(2,\C)$ associated with the complex place of $k$. The projection
$$\Gamma_\mathfrak{D} = P\rho(\mathfrak{D}^1) < \PSL(2,\C),$$
where $P: \mathrm{M}(2,\C) \to \PSL(2,\C)$, is then a discrete finite covolume subgroup of $\PSL(2,\C)$. Any subgroup of $\PSL(2,\C)$ which is commensurable with some such group $\Gamma_\mathfrak{D}$ is called an \emph{arithmetic subgroup} and the field $k$ is called its \emph{field of definition}. Every arithmetic subgroup is contained in some maximal arithmetic subgroup, and these maximal lattices can be effectively described using arithmetic invariants. One of their properties which will be particularly useful for us is that maximal arithmetic Kleinian groups are congruence subgroups of $\PSL(2,\C)$ (see \cite{LMR06}). We refer to \cite{MR03} for a detailed study of arithmetic Kleinian groups and their properties. 

The main part of the proof begins in Section~\ref{sec3} with a generalization of the Gromov--Guth inequality from \cite{GG12} to hyperbolic orbifolds. 

Let $\Orb$ be a closed orientable hyperbolic $3$-orbifold with the underlying space $|\Orb|$ and singular set $\Sigma$, which is a $3$-valent graph in $|\Orb|$. We want to define the notion of thickness for an embedding of $\Orb$ in $\R^n$ which would generalize \emph{retraction thickness} defined for manifold embeddings by Gromov and Guth. Perhaps the most straightforward generalization would be to require that both $|\Orb|$ and $\Sigma$ embed with retraction thickness $T$, however, all our attempts to work with this definition were unsuccessful because the assumptions appear to be not strong enough for proving the inequality in Theorem~\ref{thm_GG}. On the other hand, combinatorial thickness from \cite{GG12} would work better for the inequality but is too restrictive for the applications. 
We recall that an embedding $I: \cT\to\R^n$ of a simplicial complex $\cT$ has \emph{strong combinatorial thickness} $T$ if for any simplices $\Delta_1$, \ldots, $\Delta_J$ in $\cT$ the intersection $\cap_{j=1}^J N_T(I(\Delta_j))$ is non-empty if and only if $\cap_{j=1}^J \Delta_j$ is non-empty. 
The following definition combines the features of both of these notions. 

Assume that the underlying space of $\Orb$ is endowed with a triangulation $\cT$ of bounded degree (i.e. the number of simplices incident to any vertex is bounded by an absolute constant) and such that every $l$-dimensional skeleton of the singular set is contained in the $l$-skeleton of $\cT$. We say that $\Orb$ is embedded in $\R^n$ with \emph{thickness} $T$ if there is an embedding $I_1: \cT^1 \to \R^n$  with strong combinatorial thickness $T$. 
Given such an embedding we will call by the \emph{tube volume} $V_T$ the volume of the $T$-neighborhood of $I_1(\cT^1)$ in $\R^n$. It is not hard to see that assuming the dimension $n$ is sufficiently large, any thick embedding $I_1: \cT^1 \to \R^n$ extends to a topological embedding $I: \Orb \to \R^n$ with retraction thickness $T$ and volume of the $T$-neighborhood of $I(\Orb)$ bigger than $V_T$. We will elaborate more on the relation between thickness for orbifold embeddings and retraction thickness in Section~\ref{sec3.2}.



\begin{theorem} \label{thm_GG} 
Let $\Orb$ be a closed orientable hyperbolic $3$-orbifold with volume $V_{hyp}$ which has a manifold cover with the first non-zero eigenvalue of the Laplacian $\lambda_1$. If $\Orb$ is embedded in $\R^n$, $n \ge 7$ with retraction thickness $T$, we have
$$ \lambda_1^{\frac{n}{n-1}} V_{hyp}^{\frac{n}{n-1}} \le c_{n} T^{-n} V_T .$$
\end{theorem}

The proof of the theorem is based on two important inequalities: the complicated fiber inequality of Gromov and Falconer's slicing inequality. This is close to the original idea of Gromov and Guth which is described in \cite{GG12} before the proof of Theorem~3.2 but requires more care. Let us note that Theorem~3.2 of \cite{GG12} applies to the manifolds of any dimension $k \ge 3$ while our orbifold version is restricted to the $3$-dimensional case. We do not know how to generalize the result to higher dimensions and discuss this and related open problems at the end of Section~\ref{sec3}.

We are going to apply Theorem~\ref{thm_GG} to the orbifolds $\Orb = \Hy^3/\Gamma$ with $\Gamma < \PSL(2, \C)$ a maximal arithmetic subgroup and $|\Orb|$ an elliptic $3$-manifold. An important property of retraction thickness is that it is almost invariant under homotopy equivalences (see \cite[Proposition~3.2]{GG12}). An analogue of this property remains valid for orbifold thickness while considering homotopies of $2$ and $3$ dimensional skeletons of a triangulation of $\Orb$. We then combine it with a result on bad expansion of $1$-skeletons of the triangulations that we consider. 

The following Sections \ref{sec4} to \ref{sec6} deal with the bounds for $h$, $V_{hyp}$ and $V_T$.  In Section~\ref{sec4} we prove the orbifold analogues of the well known inequalities of Cheeger and Buser which relate the isoperimetric constant $h$ and the first non-zero eigenvalue $\lambda_1$ of the Laplacian on $\Orb$. These inequalities allow us to bound $\lambda_1$ in terms of $h$ and give another version of the inequality in Theorem~\ref{thm_GG} which is more similar to its manifold version in \cite{GG12}. In Section~\ref{sec5} we use arithmeticity to show that the hyperbolic volume of $\Orb$ is bounded below by $N^{1/(1+\epsilon)}$, where $N$ is the minimal size of a good bounded degree triangulation of $\Orb$ (see Theorem~\ref{thm_vol_hyp}). An important ingredient of the proof is a known number-theoretical bound towards Lehmer's conjecture about the Mahler measure of algebraic integers. In Section~\ref{sec6} we prove an upper bound for the tube volume using \emph{non-expanding} properties of $1$-skeletons of the good triangulations. 
We collect the results together and use some further properties of arithmetic groups in Section~\ref{sec7}, which finishes the proof of Theorem~\ref{thm1}. Let us note that a large part of the argument is valid or can be quickly generalized to higher dimensional hyperbolic orbifolds and other locally symmetric spaces. This, however, excludes some reductions that were done in this section and the proofs in Sections~\ref{sec3} and \ref{sec6}.

A sketch of a proof of Theorem~\ref{thm_2gen} on finiteness of $2$-generator cocompact arithmetic Kleinian groups is given in the end of Section~\ref{sec7}.

\section{The Gromov - Guth inequality} \label{sec3}

\subsection{Thick embeddings of orbifolds} In this section, we will prove Theorem~\ref{thm_GG}.

\begin{proof}
Let $\Orb = \Hy^3/\Gamma$, $\Gamma < \PSL(2, \C)$, be a closed orientable hyperbolic $3$-orbifold embedded in $\R^n$ with thickness $T$ and tube volume $V_T$. By scaling, we can assume that $T = 1$. We denote by $X^1$ the image of the embedding of the $1$-skeleton of the triangulation $\mathcal{T}$ of $\Orb$ which contains $\Sigma$ as $\Orb$ is orientable. Since $n \ge 7$, we can extend this embedding to a topological embedding $I: \Orb \to \R^n$ and let $X : = I(\Orb)$. 

By Selberg's lemma (cf. \cite[Theorem~1.3.5]{MR03}), the orbifold $\Orb$ has a smooth cover of a degree $d$. We shall denote this covering manifold by $M$ and let $\lambda_1$ be the first non-zero eigenvalue of the Laplacian on $M$. The triangulation $\cT$ lifts to a triangulation of $M$ which we denote by $\cT_1$.

We now apply the Falconer slicing theorem \cite{Fal80} to the neighborhood $N_1(X^1)$ in $\R^n$ which has volume $V_1$.  It implies that we can slice $\R^n$ by parallel hyperplanes $P_y$, $y\in\R$, so that each of them intersects $N_1(X^1)$ in surface area $\lesssim V_1^{\frac{n-1}{n}}$. We want to translate this inequality to the information about intersections of the hyperplanes with $X^1$. If we try to do it directly we notice that $X^1$ can wobble a lot which implies uncontrollably many intersections, but we can approximate it continuously to remove this effect. We will achieve that by covering $N(X^1)$ by the balls and then projecting them to $X^1$, the latter made possible by our assumption on thickness of the embedding.

Choose a maximal $1/4$-separated set of points $p_1$, $p_2$, \ldots\ in $N_{1/2}(X^1)$ and let $U = \cup_i B(p_i, 1/4)$. The argument of \cite[Lemma~3.5]{GG12} implies that each of the hyperplanes $P_y$ in our family intersects $\lesssim V_1^{\frac{n-1}{n}}$ of the balls $B(p_i, 1/4)$. As \mbox{$I_1: \cT^1 \to X^1 \subset \R^n$} is an embedding with strong combinatorial thickness $1$ and triangulation $\cT$ has bounded degree, there is a retraction $\Psi: U \to X^1$ such that the image $\Psi(B(p_i, 1/4))$ of any ball is contained in $\lesssim 1$ simplices of $\cT^1$ (cf. \cite[Proposition~3.6]{GG12}). Let $f_1$ be a map sending all points in $\Psi\left(B(p_i, 1/4)\cap P_y\right)$ to $y$.

This defines a continuous map $f_1: X^1 \to \R$ which we can extend barycentrically to a continuous map $f: X \to \R$ such that for every $y\in \R$ the preimage $f^{-1}(y)$ intersects $\cT$ in $\lesssim V_1^{\frac{n-1}{n}}$ triangles.
We have:
\begin{equation*}
\xymatrix{
M \ar@{-->}[rrd]^{g} \ar@{->}[d]^{[d]} & & \\
\Orb \ar@{->}[r]_I & X \ar@{->}[r]_f & \R
}
\end{equation*}

We now bring into picture the hyperbolic geometry of $\Orb$. The induced map \mbox{$g\!:\! M \to \R$} is continuous and we can apply to it the Gromov complicated fiber inequality \cite[Section~6.2]{Gr09} (see also \cite{Lac06} for a related result). It implies that there exists a fiber $G_y = g^{-1}(y) \subset M$ with 
\begin{equation}\label{sec3:eq1}
|\chi|_{hyp}(G_y) \ge c_1\lambda_1 V_{hyp}(M),
\end{equation}
where $|\chi|_{hyp}(G_y)$ denotes the absolute value of the sum of the Euler characteristics of all the hyperbolic connected components of the surface $G_y$ and the constant $c_1 > 0$ does not depend on $M$ or $g$.

As there are no singularities in the interior of the simplices of $f^{-1}(y)$, its triangulation lifts to a triangulation of the surface $G_y \subset M$ with $\lesssim dV_1^{\frac{n-1}{n}}$ triangles. This gives us an upper bound for the Euler characteristic $|\chi|_{hyp}$ of $G_y$:
\begin{equation}\label{sec3:eq2}
|\chi|_{hyp}(G_y) \le c_2 d V_1^{\frac{n-1}{n}}.
\end{equation}

From \eqref{sec3:eq1} and \eqref{sec3:eq2} we obtain
\begin{align*}
c_1\lambda_1(M)V_{hyp}(M) & \le c_2 d V_1^{\frac{n-1}{n}}; \\
\lambda_1(M)V_{hyp}(\Orb) & \le c V_1^{\frac{n-1}{n}},
\end{align*}
where the constant $c$ depends only on $n$. This finishes the proof of the theorem.
\end{proof}

\subsection{Relation to retraction thickness for hyperbolic manifolds} \label{sec3.2} Given a retraction thickness $T$ embedding $I: M \to \R^n$ of a closed hyperbolic $3$-manifold $M$, it should be possible to find a Delaunay triangulation restricted by $X = I(M)$ in the sense of \cite{ES97} which is embedded with strong combinatorial thickness $T$. The details of this procedure still have to be carefully checked.  If true, it would imply that our Theorem~\ref{thm_GG} gives an alternative proof of Theorem~3.2 from \cite{GG12} for the case when dimension $k = 3$. 

\subsection{Some open problems}
There remain open the questions about the higher dimensional analogues of the Gromov--Guth inequality for orbifolds. In one of the previous versions of this paper I stated the inequality for any dimension $k \ge 3$ assuming that the singular set has codimension $\ge 2$ and using a weaker notion of thickness. The idea of the proof was to generalize the proof of Theorem~3.2 in \cite{GG12} directly to orbifolds based on the orbifold Cheeger constant which is discussed in the next section. Note that the proof in \cite{GG12} does not require the complicated fiber inequality essentially substituting it by Thurston's simplex straightening method which is available for any dimension. However, in the orbifold setting it appears to be hard to control the behavior of the straightened simplices with respect to the singular locus. The complicated fiber inequality and the new definition of thickness allow us to get around this difficulty while bringing up further questions.

The first question is about the definition of thick embeddings of orbifolds. It would be desirable to avoid referring to a particular choice of a triangulation in the definition~--- compare with retraction thickness for manifold embeddings in \cite{GG12}. From this point of view we can try the following: say that a closed hyperbolic $k$-orbifold $\Orb$ is embedded in $\R^n$ with thickness $T$ if there is a topological embedding $I: \Orb \to \R^n$ with retraction thickness $T$ such that the singular CW-complex $\Sigma$ is embedded with strong combinatorial thickness $T$. Even for $k = 3$, the arguments presented above are not sufficient for proving the theorem under these weaker assumptions but I do not know any potential counterexamples that would show that such a generalization is false. 
 
More specificaly, for higher dimensions we can state two open problems. The first one goes back to Gromov's paper \cite{Gr09} (see also Appendix in \cite{GG12}):
\begin{problem}\label{prob_3.1}
Is there a higher dimensional analogue of the complicated fiber inequality?
\end{problem}
\noindent
If solved, it could potentially allow us to generalize the proof of Theorem~\ref{thm_GG} to dimensions $k >3$. The second problem is about proving the theorem without using the complicated fiber inequality, for instance, as it was done for the manifolds in \cite{GG12}:
\begin{problem}\label{prob_3.2}
Is there another proof of Theorem~\ref{thm_GG} which does not use the complicated fiber inequality and generalizes to higher dimensions or applies to a weaker notion of thickness?
\end{problem}

\section{The Cheeger and Buser inequalities for hyperbolic orbifolds} \label{sec4}

In the late 1960s, Cheeger introduced an isoperimetric constant $h$ to bound below the first eigenvalue $\lambda_1$ of the Laplacian on a closed Riemanian manifold \cite{Ch70}. Cheeger's result is that
\begin{equation*}
\lambda_1 \ge \frac{h^2}{4}.
\end{equation*}
Later on in \cite{Bus82}, Buser proved an upper bound for $\lambda_1$ in terms of $h$ and thus showed that these two invariants of Riemannian manifolds are essentially equivalent. For closed hyperbolic $k$-manifolds Buser's inequality is
\begin{equation*}
\lambda_1 \le 2(k-1)h + 10h^2.
\end{equation*}

We shall prove orbifold versions of these inequalities. 

The orbifold Cheeger constant can be defined by:
$$ h(\Orb)  = \inf \frac{\vol_{k-1}(\partial U)}{\vol_k(U)}, $$
where the infimum is taken over all open subsets $U \subset \Orb$ with the hyperbolic volume $\vol_k(U) \le \frac12 \vol_k(\Orb)$ and Hausdorff measurable boundary $\partial U$, and where $\vol_{k-1}$ denotes the $(k-1)$-dimensional induced measure on $\partial U$.
We can define the Laplace operator $-\tilde{\Delta}$ on $\Orb$ as the unique self-adjoint extension of the Laplacian (see \cite[Section~4.1]{EGM98}). Its first non-zero eigenvalue is denoted by $\lambda_1 = \lambda_1(\Orb)$. 


We expect that Gromov's complicated fiber inequality extends to the hyperbolic $3$-orbifolds but we shall not check the details here. With this inequality at hand one could prove another version of Theorem~\ref{thm_GG} which does not appeal to a manifold cover of $\Orb$ and has $\lambda_1(\Orb)$ in the place of $\lambda_1(M)$. The generalized Cheeger's inequality established below can then be applied to rewrite the bound in terms of $h(\Orb)$, and thus obtain another analogue of the Gromov--Guth inequality. These results can be also useful for attacking Problem~\ref{prob_3.2} about the higher dimensional version of inequality~\eqref{sec1:eq1}. 

\begin{theorem} \label{thm_Cheeger-Buser}
Let $\Orb$ be closed hyperbolic $k$-orbifold. Then the isoperimetric constant $h$ and the first eigenvalue $\lambda_1$ of the Laplacian on $\Orb$ satisfy the generalized Cheeger and Buser inequalities
\begin{equation*}
\frac{h^2}{4} \le \lambda_1 \le 2(k-1)h + 10h^2.
\end{equation*}
\end{theorem}

\begin{proof} The argument is based on an observation that the known analytic proofs of the Cheeger and Buser inequalities \cite{Yau75, Led94} can be generalized to orbifolds. We shall follow Ledoux's paper \cite{Led94} and indicate the necessary modifications.

We shall use the fact that the space of $C^{\infty}$ functions on $\Orb$ is dense in $L^2(\Orb)$ (see \cite[Section~4.1]{EGM98}). Let $\mu$ denote the normalized Riemannian measure on $\Orb$.

For Cheeger's inequality we begin the argument as in \cite{Led94}, noting that the definition of $h$ together with the coarea formula imply
\begin{equation}\label{sec4:eq1}
 h \int_0^\infty \min\big( \mu(g \ge s), 1-\mu(g \ge s)\big) ds \le \int |\nabla g| d\mu
\end{equation}
for every positive smooth $g$ on $\Orb$. Given a smooth function $f$ on $\Orb$, let $m$ be a median of $f$ for $\mu$, i.e. $\mu(f \ge m) \ge \frac12$ and $\mu(f \le m) \ge \frac12$, and let $f^+ = \max(f-m, 0)$, $f^- = - \min(f-m, 0)$. Applying \eqref{sec4:eq1} to $g = (f^+)^2$ and $g = (f^-)^2$ together with integration by parts and followed by the Cauchy-Schwarz inequality, we get
$$ \frac{h^2}{4} \int |f-m|^2 d\mu \le \int |\nabla f|^2 d\mu.$$ 
Since the mean $\int f^2 d\mu$ for $f$ with $\int f d\mu = 0$ minimizes $\int |f-c|^2 d\mu$, $c\in \R$, we obtain  
$$ \frac{h^2}{4} \le \frac{\int |\nabla f|^2 d\mu}{\int f^2 d\mu}$$ 
for any smooth $f$ on $\Orb$ with $\int f d\mu = 0$. We now recall that the Rayleigh–-Ritz principle for semi-bounded below operators and the density of $C^{\infty}$ functions in $L^2(\Orb)$ imply that $\lambda_1(\Orb)$ is equal to the infimum of the right hand side in the last inequality (see \cite[Proof of Lemma~6.1]{ABSW} for more details). This proves the first inequality of the theorem.

The proof of Buser's inequality follows similarly. Our assumptions on $U$ and $\partial U$ imply that its characteristic function $\chi_U$ can be approximated in the $L^1$-norm by smooth functions. This is the only modification required, the rest of the proof follows the method of Ledoux which is based on the Li--Yau inequality and results of Varopoulos (see \cite[Proof of Theorem~1]{Led94}). For the characterization of $\lambda_1$ here again we make use of the density of the smooth functions in the domain of $-\tilde{\Delta}$.
\end{proof}

As in the smooth case, it is possible to generalize the inequalities in Theorem~\ref{thm_Cheeger-Buser} to the case of complete non-compact hyperbolic $k$-orbifolds. The details are entirely similar.

\section{A lower bound for hyperbolic volume} \label{sec5}

We now turn back to the three-dimensional case. The goal of this section is to give a lower bound for the volume of a closed arithmetic hyperbolic $3$-orbifold in terms of its  triangulations. In general such a bound does not exist because there are examples of closed hyperbolic $3$-manifolds of bounded volume and arbitrarily large complexity. We can achieve a desired result by exploiting arithmeticity which allows us to effectively bound the injectivity radius and the orders of singularities in terms of volume.

We will call a triangulation of $|\Orb|$ by a \emph{good triangulation} of $\Orb$ if any $k$-dimensional submanifold of its singular set is contained in the $k$-skeleton of the triangulation (for $k = 0, \ldots, \dim(\Sigma)$).

We shall first consider the problem of bounding the orbifold volume under the assumptions that the injectivity radius $r_{inj} \ge r > 0$ and that for any finite subgroup $G < \Gamma = \pi_1(\Orb)$ we have $|G| \le m$. A similar problem was studied before by Gelander and Samet, who were considering the volume of a thick part of an orbifold and used the Margulis constant as a uniform lower bound for the injectivity radius (see \cite[Section~2]{BGLS} and \cite{Sam13} for the details). In particular, Samet's Theorem~4.2 specialized to the hyperbolic $3$-orbifolds is very close to what we need except that we require an explicit control over the constants and we want to obtain not just a simplicial complex homotopy equivalent to $\Orb$ but an actual triangulation of the orbifold.

Let us first recall some definitions. Let $\Orb = X/\Gamma$ with $X = \Hy^3$ be a compact orientable hyperbolic $3$-orbifold with the singular set $\Sigma$ and $\pi: X \to \Orb$ is the canonical projection. We shall call the isometries in $\Gamma$ which act on $X$ without fixed points \emph{hyperbolic}. For a hyperbolic isometry $\gamma$ its \emph{displacement at $x \in X$} is defined by $\ell(\gamma, x) = \dist(\gamma x, x)$ and the \emph{displacement of $\gamma$} (also called its \emph{translation length} $\ell(\gamma)$) is given by the displacement of $\gamma$ at the points of its axis (cf. \cite[Section~12.1]{MR03}). We shall define the orbifold \emph{injectivity radius} by $r_{inj}(\Orb) = \inf \{\frac12 \ell(\gamma)\}$, where the infimum is taken over all hyperbolic elements $\gamma \in \Gamma$. It is equal to the half of the smallest length of a closed geodesic in $\Orb$. When $\Orb$ is a manifold this definition is equivalent to the usual definition of the injectivity radius as the supremum of $r$ such that any point $p\in\Orb$ admits an embedded ball $B(p, r) \subset \Orb$. This is not the case in general as the points around the singular set only admit so called folded balls, but this fact will not cause us any serious problems.


\begin{prop}\label{prop_vol_hyp}
Let $\Orb$ be a closed orientable hyperbolic $3$-orbifold of volume $V_{hyp}$ with $r_{inj} \ge r > 0$ and  $|G| \le m$ for any finite subgroup $G < \Gamma$. Then $\Orb$ has a good triangulation with $N$ simplices for
$$N \le \frac{mV_{hyp}}{v_{\epsilon/2}}(D-1),$$
where $\epsilon = \frac{1}{10}\min\{r, \mu_3\}$, $\mu_3$ is the Margulis constant of $\Hy^3$, $v_{\epsilon/2}$ is the volume of a ball of radius $\epsilon/2$ in $\Hy^3$, and the degree $D$ of the vertices of the triangulation is bounded above by $\frac{v_{5\epsilon/2}}{v_{\epsilon/2}}$.
\end{prop}

\begin{proof}
Since $\epsilon < \mu_3$,  the distance between any two vertices of the singular set is bigger than $\epsilon$. Indeed, otherwise the subgroup of $\Gamma$ generated by the stabilizers of the vertices would be abelian, which is impossible. A similar argument implies that for any two different points $x,y\in \Sigma$, if $\dist(x,y) < \epsilon$ then $\dist(x,y) = \dist_\Sigma(x,y)$ or there exists a vertex in $\Sigma$ which is $\epsilon$-close to both $x$ and $y$ (we shall call it by the property that \emph{the singular set does not  come too close to itself}).

Let us define a maximal $\epsilon$-separated set $P$ of points $p_1, p_2, \ldots $ in $\Orb$ which contains all the vertices of $\Sigma$ as well as a maximal $\epsilon$-separated set of points in $\Sigma$ (in general such a set can be constructed using induction by the dimension of the submanifolds in $\Sigma$). A collection of open (folded) balls $\mathcal{B}  = \{B(p_i, \epsilon)\}$ defines a cover of $\Orb$, moreover, these balls and their intersections are contractible (this is obvious in dimension $3$, for a general proof see \cite[Propositions~4.9, 4.10]{Sam13}). Therefore the nerve of the covering defines a simplicial complex homotopy equivalent to $|\Orb|$. Moreover, the Delaunay triangulation of $P$ is a geodesic triangulation of $\Orb$ which satisfies the properties of a good triangulation. In order to construct the triangulation first note that the assumption that $\epsilon\le\frac{r_{inj}}{10}$ allows us to work on the universal covering $\Hy^3$. The vertices of the singular set are among the vertices of the triangulation and the edges of $\Sigma$ are geodesic and covered by the $\epsilon$-net of vertices so they are contained in the $1$-skeleton. We remark that the vertices of $\Sigma$ are $>2\epsilon$ separated, hence any sphere containing two such vertices will contain some other point from $P$ in its interior, and all other points in $P$ can be moved slightly, if necessary, in order to assume that the set $P$ is generic and hence the Delaunay triangulation is well defined. (We refer to \cite{Bre09} for a more thorough discussion of Delaunay triangulations of hyperbolic manifolds.)

Now consider smaller balls $B(p_i, \epsilon/2)$ in $\Orb$. These balls are disjoint and because the singular set cannot come too close to itself, each of them is isometric to a quotient of an $\epsilon/2$-ball in $\Hy^3$ by a finite subgroup of $\Gamma$ (which could be trivial, of course). Hence $\vol(B(p_i, \epsilon/2)) \ge \frac{v_{\epsilon/2}}{m}$ which implies $|\mathcal{B}| \le \frac{mV_{hyp}}{v_{\epsilon/2}},$ the upper bound for the number of vertices in the triangulation. 

To bound the degree of the vertices note that in order for two vertices to be joined by an edge in the Delaunay triangulation they have to be at the distance less than $2\epsilon$ from each other. So assume that a ball $B(p_i, \epsilon)$ intersects $D$ other balls in $\mathcal{B}$. Then all of them are contained in $B(p_i, 3\epsilon)$, and the corresponding $\epsilon/2$-balls are all in $B(p_i, 5\epsilon/2)$. It follows that the number of balls in $\mathcal{B}$ which intersect $B(p_i, \epsilon)$ is at most $\frac{v_{5\epsilon/2}}{v_{\epsilon/2}}$, which gives the upper bound for the vertex degree.

Finally, applying the generalized Dehn--Sommerville equations for the closed triangulated manifold $|\Orb|$, we obtain that the number of simplices in the triangulation is bounded above by $|\mathcal{B}|(D-1)$.
\end{proof}

\begin{rmk}
 If we assume the Short Geodesic Conjecture for arithmetic hyperbolic $3$-orbifolds, which says that there exists a uniform positive lower bound for their injectivity radii (cf. \cite[Section~12.3]{MR03}), then the bound in the proposition restricted to arithmetic orbifolds $\Orb$ would not depend on $r$ but the dependence on $m$ would still remain.
\end{rmk}

We now recall some properties of arithmetic groups which will allow us to relate the orders of singularities and the injectivity radius of the quotient orbifold to its volume.

Assume that $\gamma \in \Gamma$ has the smallest displacement and let $P(x)$ denote its minimal polynomial. Then we have $r_{inj}(\Orb) = \frac12 \ell(\gamma) = \frac12 \log M(P)$ or $\log M(P)$ according to whether $\mathrm{tr}(\gamma)$ is complex or real, where $M(P) = \prod_{i=1}^d \mathrm{max}(1, |\theta_i|)$, $\theta_1$, \dots, $\theta_d$ are the roots of $P(x)$, is the polynomial \emph{Mahler measure} (for more details see \cite[Chapter~12]{MR03}).

The celebrated Lehmer's problem says that the Mahler measures of non-cyclotomic polynomials are expected to be uniformly bounded away from $1$. A special case of this conjecture is known to be equivalent to the Short Geodesic Conjecture for arithmetic hyperbolic $3$-orbifolds. These conjectures have attracted a lot of interest but still remain wide open. Nevertheless, there are some quantitative number-theoretic results towards Lehmer's problem which can be very useful.

Let us recall a well known Dobrowolski's bound for the Mahler measure:
\begin{equation}\label{sec5:eq1}
\log M(P) \ge c_1\left(\frac{\log\log d}{\log d}\right)^3,
\end{equation}
where $d$ is the degree of the polynomial $P$ and $c_1>0$ is an explicit constant.

Next we claim that the field of definition $k$ of the arithmetic group $\Gamma$ satisfies
\begin{equation*}
d \le 2\mathrm{deg}(k).
\end{equation*}
This indeed follows immediately from the results discussed in~\cite[Chapter~12]{MR03}.

We can relate these quantities to the volume by using an important inequality relating the volume of an arithmetic hyperbolic $3$-orbifold and the degree of its field of definition:
\begin{equation}\label{sec5:eq2}
\mathrm{deg}(k) \le c_2\log\vol(\Orb) + c_3.
\end{equation}
This result was first proved by Chinburg and Friedman \cite{CF86}, in a form stated here it can be found in~\cite{BGLS}. We note that the constants in inequalities \eqref{sec5:eq1}--\eqref{sec5:eq2} can be computed explicitly.

For sufficiently large $x$ the function $\frac{\log x}{x}$ is monotonically decreasing, hence for sufficiently large volume we obtain
\begin{equation}\label{sec5:eq3}
r_{inj}(\Orb) \ge \frac{c_1}2\left(\frac{\log\log\log \vol(\Orb)^c}{\log\log \vol(\Orb)^c}\right)^3.
\end{equation}

We note that this is a very slowly decreasing function. For more information about this argument and some related results we refer to \cite{Bel10}.

It remains to bound the order of finite subgroups $G < \Gamma$. 
Finite subgroups of a Kleinian group are the discrete subgroups of $\mathrm{SO}(3)$, the group of orientation preserving isometries of the $2$-sphere. These are the cyclic, dihedral, tetrahedral, octahedral and icosahedral symmetry groups. Their orders are bounded above by a constant multiple of the order of the maximal cyclic subgroup. From the other hand, an arithmetic Kleinian group $\Gamma$ defined over a field $k$ can contain a cyclic subgroup of order $n$ only if $\cos(2\pi/n) \in k$. Hence we have $\phi(n) \le \deg(k)$, where the Euler function $\phi$ satisfies the well known inequality $\sqrt{n}/2 \le \phi(n)$. Bringing these observations together, we obtain that $m = |G|$ satisfies the inequalities
\begin{equation}\label{sec5:eq4}
m \le c_4\deg(k)^{2} \le c_5(\log\vol(\Orb))^{2}.
\end{equation}

Finally, for sufficiently small positive $\epsilon$ we have

\begin{equation}\label{sec5:eq5}
\frac{v_{5\epsilon/2}}{v_{\epsilon/2}} = \frac{\pi(\sinh(5\epsilon)-5\epsilon)}{\pi(\sinh(\epsilon)-\epsilon)} < 125.1.
\end{equation}
This gives an upper bound for the degree of the vertices of the triangulation. This bound can be significantly improved by a more careful consideration but we shall not pursue it here.

\medskip

From Proposition~\ref{prop_vol_hyp} and inequalities \eqref{sec5:eq3}--\eqref{sec5:eq5} we deduce

\begin{theorem} \label{thm_vol_hyp}
For any $\epsilon > 0$, there is a constant $V_0 = V_0(\epsilon)$ such that any closed orientable arithmetic hyperbolic $3$-orbifold of volume $V_{hyp} \ge V_0$ has a good triangulation with at most $V_{hyp}^{1+\epsilon}$ simplices and vertex degree bounded above by an absolute constant.
\end{theorem}

\section{An upper bound for tube volume} \label{sec6}
The goal of this section is to give an upper bound for the tube volume of an embedding of a hyperbolic $3$-orbifold $\Orb$ assuming that the underlying space of $\Orb$ is an elliptic $3$-manifold and that $\Orb$ admits a good triangulation of bounded degree with at most $N$ simplices.

An elliptic $3$-manifold can be presented as a quotient of $\Sph^3$ by a free action of a finite group $G$. These manifolds and their groups were classified by Hopf and Seifert--Threlfall \cite{H26, ST31, ST33}. In particular, the group $G$ is either cyclic, or is a central extension of a dihedral, tetrahedral, octahedral, or icosahedral group by a cyclic group of even order. The corresponding manifolds are called lens spaces, prism manifolds, and then tetrahedral, octahedral and icosahedral manifolds, respectively.

\begin{theorem} \label{thm_vol_tube}
Suppose that the underlying space of an orbifold $\Orb$ is an elliptic \mbox{$3$-manifold} and that $\Orb$ has a good triangulation $\cT$ into $N$ simplices with vertex degree bounded above by $D$. Then for sufficiently large $n$, the orbifold $\Orb$ admits an embedding into $\R^n$ with thickness $1$ and tube volume $V_{1} \le c_{D, n}N$.
\end{theorem}

\begin{proof}[Sketch of a proof:]
 
%
%
%

Let $X_0$ be a connected based subcomplex of $\cT^1$ which carries the fundamental group of $|\Orb|$: for the sphere it is the base point, for a lens space it is a cycle of length $\lesssim N$, and for the other types it is a bouquet of $\le c_1 = \mathrm{const}$ cycles of length $\lesssim N$. The subcomplex $X_0$ has two properties that we would like to note:
\begin{itemize}
 \item[(1)] The complex $\cT$ admits an $L_D$-exhaustion with respect to $X_0$ in the sense of \cite[Section~4.4]{Gr10};
 \item[(2)] For any $n_1\ge 2c_1$, there is an embedding of $X_0$ in $\R^{n_1}$ with strong combinatorial thickness $1$ and tube volume $\sim N$ (we can embed a cycle to a circle in $\R^2$ and use the extra dimensions for any extra cycle whose total number is bounded by $c_1$). 
\end{itemize}

We can now start with the embedding from (2) and extend it to an embedding of $\cT^1$ to $\R^n$ using property (1) together and the boundedness of degree of $\cT^1$. The dimension $n$ is bounded below by a constant depending on $n_1$ and $D$. 

The details of this argument will be given in a subsequent version.
\end{proof}

\section{Completion of the proofs of the main results} \label{sec7}

Let $\Gamma$ be a maximal arithmetic Kleinian group as in Section~\ref{sec2} so the underlying space of the orbifold $\Orb = \Hy^3/\Gamma$ is an elliptic $3$-manifold.

Let $\epsilon = 0.01$ and $V_0 = V_0(\epsilon)$ is the constant from Theorem~\ref{thm_vol_hyp}. If $V_{hyp}(\Orb) \ge V_0$, by the theorem $\Orb$ admits a good triangulation with $N \le V_{hyp}^{1+\epsilon}$ simplices and vertex degree bounded by $D$ ($=125$). Now Theorem~\ref{thm_vol_tube} implies that $\Orb$ admits an embedding into $\R^n$ with retraction thickness $1-\epsilon$ and tube volume 
$$ V_{1-\epsilon} \le c_{D,n} N \le c_{D,n} V_{hyp}^{1+\epsilon}.$$

We would like to apply the generalized Gromov--Guth inequality from Theorem~\ref{thm_GG}, which requires also the information about the spectral gap $\lambda_1$ of a manifold cover of $\Orb$. We recall that the maximal arithmetic Kleinian groups are congruence (cf. \cite{LMR06}), hence $\Gamma$ contains a principal congruence subgroup $\Gamma(I)$ for some ideal $I$ in the ring of integers $R$ of the field of definition of $\Gamma$. By Selberg's lemma, $\Gamma$ has a torsion-free subgroup of finite index of the form $\Gamma(J)$ for some ideal $J\subset R$. It follows that $\Gamma(I\cap J)$ is a torsion-free congruence subgroup of $\Gamma$, and we can take $M = \Hy^3/\Gamma(I\cap J)$. The results towards Selberg's and Ramanujan's conjectures about the spectrum of the Laplacian apply to $M$, and by deep work of Vigneras and Burger--Sarnak we know that $\lambda_1(M) \ge \frac34$ \cite{Vig83, BS91}. 

We are now ready to consider the inequality from Theorem~\ref{thm_GG}:
\begin{align*}
\lambda_1^{\frac{n}{n-1}} V_{hyp}^{\frac{n}{n-1}} & \le c_{n}c_{D,n} V_{hyp}^{1+\epsilon};\\
V_{hyp}^{\frac{n}{n-1} - (1+\epsilon)} & \le c_{n}c_{D,n}\lambda_1^{-\frac{n}{n-1}}.
\end{align*}
By Theorem~\ref{thm_vol_tube}, we can assume that $n = 30$, so $\frac{n}{n-1} - (1+\epsilon) > 0.02$, and we deduce that 
\begin{align*}
V_{hyp} & \le C, 
\end{align*}
for an absolute positive constant $C$ which can be computed explicitly.

Now Borel's finiteness theorem \cite{Bor81} says that there exist only finitely many (up to conjugacy) arithmetic Kleinian groups of bounded covolume $V_{hyp} \le \max(V_0, C)$, so there are only finitely many possibilities for  $\Gamma$. The group $\Gamma$ uniquely determines its maximal subgroup generated by elements of finite order $\Gamma_0$, hence we obtain finiteness of the number of such subgroups. This finishes the proof of Theorems~\ref{thm2.1} and \ref{thm1}. \qed

\medskip

We conclude with another finiteness result discussed in the introduction.

\begin{theorem}\label{thm_2gen}
There are only finitely many conjugacy classes of cocompact $2$-genera\-tor arithmetic Kleinian groups. 
\end{theorem}

\begin{proof}[Sketch of a proof:]
As in the proof of Corollary~\ref{cor2}, we recall that any $2$-generator Kleinian group is contained with index at most $2$ in a Kleinian group generated by three half-turns and hence by Corollary~\ref{cor1} there are only finitely many commensurability classes of such cocompact arithmetic subgroups. It remains to show that any commensurability class can contain only finitely many $2$-generator groups. This can be done by adapting the methods from \cite{BS11} to orbifolds. 

Consider a sequence of commensurable closed arithmetic $3$-orbifolds $\Orb_1$, $\Orb_2$, \ldots with $\rank(\pi_1(\Orb_i)) = 2$. Their orbifold injectivity radii (defined as in Section~\ref{sec5}) are uniformly bounded below. By Selberg's lemma we have a sequence of covering manifolds $M_i \to \Orb_i$, moreover, the degrees of the covers are uniformly bounded by a constant so after passing to a subsequence we can assume that they are all equal to $d$. By \cite[Proposition 6.1]{BS11}, we can assume that $(M_i)$ converges in the based Gromov--Hausdorff topology to a manifold $M_\infty$ that is homeomorphic to $\Sigma_M\times\R$ and has two degenerate ends. By the construction, the corresponding pointed Gromov--Hausdorff limit $\Orb_\infty$ of the sequence $(\Orb_i)$ is $d$-fold covered by $M_\infty$, hence by Agol's version of the Thurston--Canary covering theorem \cite[Lemma~14.3]{Ag04} the orbifold $\Orb_\infty$ is also a product $\Sigma\times\R$, where $\Sigma$ is a closed hyperbolic $2$-orbifold.

Repeating the first part of the proof of Theorem~7.3 from \cite{BS11} we can conclude that there is a subsequence $(\Orb_i)$ such that $\Orb_\infty$ covers $\Orb_i$ and for any compact subset $K \subset \Orb_\infty$ the covering is injective for large enough $i$. We now need an orbifold version of the covering theorem from \cite[Appendix]{BS11}, which implies that all but finitely many $\Orb_i$ fiber over $\Sph^1$ or $\Sph^1/(z\to\bar{z})$ with regular fibers homeomorphic to $\Sigma$. This result can be proved by a combination of the methods of Canary, Agol, and Biringer--Souto \cite{Can96, Ag04, BS11} but we shall not present the details here. Finally, we require a generalization to orbifolds of a result of White, Biringer, Souto which implies that for all but finitely many $\Orb_i$ we have $\rank(\pi_1(\Orb_i)) = \rank(\pi_1(\Sigma)) + 1$ or $\rank(\pi_1(\Orb_i)) = \rank(\pi_1(\Sigma^s)) + 1$, where $\Sigma^s$ denotes the singular fiber of the fibration over $\Sph^1/(z\to\bar{z})$. See e.g. \cite[Chapter~4]{Bir09} and the references therein for the manifold version of this theorem. The orbifold version can be proved by similar methods but requires some care. Both possibilities immediately lead to a contradiction, hence there are no such infinite sequences of orbifolds.
\end{proof}

\medskip

\noindent\textbf{Acknowledgements.} I would like to thank Hannah Alpert and Larry Guth for pointing out the gaps in my previous versions of the proof of Theorem~\ref{thm_GG}. 



\begin{thebibliography}{xxxx}

\bibitem[Ago04]{Ag04} I. Agol, Tameness of hyperbolic $3$-manifolds, {\sl arXiv:math/0405568}, 22~pp.

\bibitem[Ago06]{Ag06} I. Agol, Finiteness of arithmetic Kleinian reflection groups, {\em in: Proceedings of the International Congress of Mathematicians, Madrid, 2006}, Vol. II, Eur. Math. Soc., Zürich, 2006, pp. 951--960. 

\bibitem[ABSW08]{ABSW} I. Agol, M. Belolipetsky, P. Storm, K. Whyte, Finiteness of arithmetic hyperbolic reflection groups, {\em Groups, Geometry, and Dynamics} {\bf 2} (2008), 481--498.

\bibitem[Arm68]{Arm68} M. A. Armstrong, The fundamental group of the orbit space of a discontinuous group, {\em Proc. Cambridge Philos. Soc.} {\bf 64} (1968), 299--301.

\bibitem[Bel02]{Bel02} M. Belolipetsky, Singular sets and parameters of generalized triangle orbifolds, {\sl arXiv:math.MG/0103015}, 21~pp.

\bibitem[Bel10]{Bel10} M. Belolipetsky, Geodesics, volumes and Lehmer's conjecture, {\em Oberwolfach Rep.} {\bf 7} (2010) 2136--2139.

\bibitem[BGLS10]{BGLS} M. Belolipetsky, T. Gelander, A. Lubotzky, A. Shalev, Counting arithmetic lattices and surfaces, {\em Ann. of Math.} {\bf 172} (2010) 2197--2221.

\bibitem[Bir09]{Bir09} I. Biringer, {\em Algebra vs. geometry in hyperbolic $3$-manifolds}, PhD thesis, University of Chicago, 2009. 

\bibitem[BS11]{BS11} I. Biringer and J. Souto, A finiteness theorem for hyperbolic 3-manifolds, {\em J. Lond. Math. Soc. (2)} {\bf  84} (2011), 227--242. 

\bibitem[Bor81]{Bor81} A. Borel, Commensurability classes and volumes of hyperbolic 3-manifolds, {\em Ann. Scuola Norm. Sup. Pisa (4)} {\bf 8} (1981), 1--33.

\bibitem[Bre09]{Bre09} W. Breslin, Thick triangulations of hyperbolic $n$-manifolds. {\em Pacific J. Math.} {\bf 241} (2009), 215--225.

\bibitem[BS91]{BS91} M. Burger and P. Sarnak, Ramanujan duals. II, {\em Invent. Math.}  {\bf 106} (1991), 1--11.

\bibitem[Bus82]{Bus82} P. Buser, A note on the isoperimetric constant, {\em Ann. Sci. École Norm. Sup.} {\bf 15} (1982), 213--230.


\bibitem[Can96]{Can96} R. D. Canary, A covering theorem for hyperbolic $3$-manifolds and its applications, {\em Topology} {\bf 35} (1996), 751--778. 

\bibitem[Che70]{Ch70} J. Cheeger, A lower bound for the smallest eigenvalue of the Laplacian, Problems in Analysis, Symposium in honor of S. Bochner, Princeton Univ. Press, Princeton, NJ, 1970, pp. 195--199.

\bibitem[CFG92]{CFG92} J. Cheeger, K. Fukaya, M. Gromov, Nilpotent structures and invariant metrics on collapsed manifolds, {\em J. Amer. Math. Soc.} {\bf 5} (1992), 327--372.

\bibitem[CF86]{CF86}  T. Chinburg and E. Friedman, The smallest arithmetic hyperbolic three-orbifold, {\em Invent. Math.} {\bf 86} (1986), 507--527.

\bibitem[CMMB02]{CMMB02} M. D. E. Conder, C. Maclachlan, G. Martin, E. A. O'Brien, $2$-generator arithmetic Kleinian groups. III, {\em Math. Scand.} {\bf 90} (2002), 161--179.

\bibitem[DS73]{DS73} M. N. Dyer and A. J. Sieradski, Trees of homotopy types of two-dimensional CW-complexes, {\em Comm. Math. Helv.} {\bf 48} (1973), 31--44.

\bibitem[ES97]{ES97} H. Edelsbrunner and N.R. Shah, Triangulating topological spaces, {\em Int. J. Comput. Geom. Appl.} {\bf 7} (1997) 365--378.

\bibitem[EGM98]{EGM98} J. Elstrodt, F. Grunewald, J. Mennicke, {\em Groups acting on hyperbolic space}, Springer Monogr. Math., Springer, Berlin 1998.

\bibitem[Fal80]{Fal80} K. J. Falconer, Continuity properties of k-plane integrals and Besicovitch sets, {\em Math. Proc. Cambridge Phil. Soc.} {\bf  87} (1980), 221--226.


\bibitem[GM94]{GM94} F. Gehring and G. Martin, Commutators, collars and the geometry of M\"obius groups, {\em J.~d'Analyse Math.} {\bf 63} (1994), 175--219.


\bibitem[GMMR97]{GMMR97} F. Gehring, C. Maclachlan, G. Martin, A. Reid, Arithmeticity, discreteness and volume, {\em Trans. Am. Math. Soc.} {\bf 349} (1997), 3611--3643.

\bibitem[GMM98]{GMM98} F. Gehring, C. Maclachlan, G. Martin, Two-generator arithmetic Kleinian groups II, {\em Bull. Lond. Math. Soc.} {\bf 30} (1998), 258--266.


\bibitem[Gil97]{Gil97} J. Gilman, A discreteness condition for subgroups of $\PSL(2,\C)$, {\em Contemp. Math.} {\bf 211} (1997), 261--267.

\bibitem[Gro09]{Gr09} M. Gromov, Singularities, expanders and topology of maps. Part 1: Homology versus volume in the spaces of cycles, {\em Geom. Funct. Anal.} {\bf 19} (2009), 743--841.

\bibitem[Gro10]{Gr10} M. Gromov, Singularities, expanders and topology of maps. Part 2: From combinatorics to topology via algebraic isoperimetry, {\em Geom. Funct. Anal.} {\bf 20} (2010), 416--526.

\bibitem[GG12]{GG12} M. Gromov and L. Guth, Generalizations of the Kolmogorov--Barzdin embedding estimates, {\em Duke Math. J.} {\bf 161} (2012), no. 13, 2549--2603.



\bibitem[Hen77]{Hen77} H. Hendriks, Obstruction theory in $3$-dimensional topology: an extension theorem, {\em J. London Math. Soc.} {\bf 16} (1977), 160--164; Corrigendum, {\em Ibid.} {\bf 18} (1978), 192.

\bibitem[HLM91]{HLM91} H. Hilden, M. T. Lozano, J. M. Montesinos-Amilibia, On the arithmetic $2$-bridge knots and link orbifolds and a new knot invariant, {\em J. Knot Theory Ramifications} {\bf 4} (1995), 81--114. 

\bibitem[Hopf26]{H26} H. Hopf, Zum Clifford-Kleinschen Raumproblem, {\em Math. Ann.} {\bf 95} (1926) 313--339.


\bibitem[Kli89]{Kli89} E. Klimenko, Discrete groups in three-dimensional Lobachevsky space generated by two rotations, {\em Sib. Math. J.} {\bf 30} (1989), 95--100.

\bibitem[KK05]{KK05} E. Klimenko and N. Kopteva, All discrete ${\mathcal RP}$ groups whose generators have real traces, {\em Internat. J. Algebra Comput.} {\bf 15} (2005), 577--618. 

\bibitem[Lac06]{Lac06} M. Lackenby, Heegaard splittings, the virtually Haken conjecture and property ($\tau$), {\em Invent. Math.}, {\bf 164} (2006), 317--359.

\bibitem[Led94]{Led94} M. Ledoux, A simple analytic proof of an inequality by P.~Buser, {\em Proc. Amer. Math. Soc.} {\bf 121} (1994), no. 3, 951--959.

\bibitem[LMR06]{LMR06} D. Long, C. Maclachlan, A. Reid, Arithmetic Fuchsian groups of genus zero, \emph{Pure Appl.~Math.~Q.} \textbf{2} (2006), no.~2, 569--599.

\bibitem[MM99]{MM99} C. Maclachlan and G. Martin, $2$-generator arithmetic Kleinian groups, {\em J. Reine Angew. Math.} {\bf 511} (1999), 95--117.

\bibitem[MR03]{MR03} C. Maclachlan and A. W. Reid, {\em The arithmetic of hyperbolic 3-manifolds}, Grad. Texts Math. {\bf 219}, Springer, New York, 2003.


\bibitem[MP01]{MP01} S. Matveev and E. Pervova, Lower bounds for the complexity of three-dimensional manifolds, {\em Dokl. Akad. Nauk} {\bf 378} (2001), no. 2, 151--152.


\bibitem[MT77]{MT77} J. Milnor and W. Thurston, Characteristic numbers of 3-manifolds, {\em Enseign. Math. (2)} {\bf 23} (1977), 249--254.

\bibitem[Nash56]{Nash} J.~Nash, The imbedding problem for Riemannian manifolds, {\em Ann. of Math. (2)} {\bf 63} (1956), 20--63. 

\bibitem[Nik91]{Nik91} I.~Nikolaev, Bounded curvature closure of the set of compact Riemannian manifolds, {\em Bull. Amer. Math. Soc.} {\bf 24} (1991), 171--177.

\bibitem[Roos17]{Roos17} S.~Roos, A characterization of codimension 1 collapse under bounded curvature and diameter, {\sl arXiv:1701.06515v1}, 18~pp.

\bibitem[Sam13]{Sam13} I. Samet, Betti numbers of finite volume orbifolds, {\em Geom. Topol.} {\bf 17} (2013), no. 2, 1113--1147.

\bibitem[ST31]{ST31} H. Seifert and W. Threlfall, Topologische Untersuchung der diskontinuit\"atsbereiche endlicher Bewegungsgruppen des dreidimensionalen sph\"arischen Raumes, {\em Math. Ann.} {\bf 104} (1931), 1--70.

\bibitem[ST33]{ST33} H. Seifert and W. Threlfall, Topologische Untersuchung der diskontinuit\"atsbereiche endlicher Bewegungsgruppen des dreidimensionalen sph\"arischen Raumes (Schlus), {\em Math. Ann.} {\bf 107} (1933), 543--586.

\bibitem[Sul87]{Sul87} D.~Sullivan, Related aspects of positivity in Riemannian geometry, {\em J. Differential Geom.} {\bf 25} (1987), 327--351.

\bibitem[Vig83]{Vig83} M.-F. Vign\'eras, Quelques remarques sur la conjecture $\lambda_1 \ge \frac14$, in: {\em Seminar on number theory, Paris 1981--82 (Paris, 1981/1982), Progr. Math.,} {\bf 38}, Birkh\"auser Boston, 1983, 321--343.

\bibitem[Yau75]{Yau75} S.-T. Yau, Isoperimetric constants and the first eigenvalue of a compact Riemannian manifold, {\em Ann. Sci. École Norm. Sup.} {\bf 8} (1975), 487--507.

\bibitem[Weyl39]{Weyl} H.~Weyl, On the volume of tubes, {\em Amer. J. Math.} {\bf 61} (1939) 461--472. 


\end{thebibliography}
\end{document}